\documentclass[a4paper,12pt,reqno]{amsart}
\usepackage{amsfonts}
\usepackage{amsmath}
\usepackage{amssymb}
\usepackage[a4paper]{geometry}
\usepackage{mathrsfs}
\usepackage[english]{babel}
\usepackage{xcolor}
\usepackage{enumerate}
\usepackage[colorlinks]{hyperref}
\renewcommand\eqref[1]{(\ref{#1})} 

\numberwithin{equation}{section}
\theoremstyle{plain}
\newtheorem{theorem}{Theorem}[section]
\newtheorem{proposition}[theorem]{Proposition}
\newtheorem{corollary}[theorem]{Corollary}

\theoremstyle{definition}

\newtheorem{example}[theorem]{Example}

\newcommand{\C}{{\mathbb C}}

\newcommand{\N}{{\mathbb N}}
\newcommand{\R}{{\mathbb R}}
\newcommand{\h}{{\mathbb H}}
\newcommand{\Rn}{{\R}^n}

\newcommand{\slp}{{\mathcal L}}
\newcommand{\Rep}{{\rm Re}}
\newcommand{\Imp}{{\rm Im}}
\newcommand{\HS}{{\mathtt{HS}}}

\begin{document}

\title{Time-dependent wave equations on graded groups}

\author[M. Ruzhansky]{Michael Ruzhansky}
\address{
  Michael Ruzhansky:
  \endgraf
  Department of Mathematics
  \endgraf
  Imperial College London
  \endgraf
  180 Queen's Gate, London SW7 2AZ
  \endgraf
  United Kingdom
  \endgraf
  {\it E-mail address} {\rm m.ruzhansky@imperial.ac.uk}
  }
\author[C. Taranto]{Chiara Taranto}
\address{
  Chiara Taranto:
  \endgraf
  Department of Mathematics
  \endgraf
  Imperial College London
  \endgraf
  180 Queen's Gate, London SW7 2AZ
  \endgraf
  United Kingdom
  \endgraf
  {\it E-mail address} {\rm c.taranto13@imperial.ac.uk}
  }

\thanks{The first author was supported in parts by the EPSRC
 grant EP/K039407/1 and by the Leverhulme Grant RPG-2014-02. No new data was collected or generated during the course of research.}


     \begin{abstract}
    In this paper we consider the wave equations for hypoelliptic homogeneous left-invariant operators on graded Lie groups with time-dependent H\"older (or more regular) non-negative propagation speeds. The examples are the time-dependent wave equation for the sub-Laplacian on the Heisenberg group or on general stratified Lie groups, or $p$-evolution equations for higher order operators on $\Rn$ or on groups, already in all these cases our results being new. We establish sharp well-posedness results in the spirit of the classical result by Colombini, de Giorgi and Spagnolo. In particular, we describe an interesting local loss of regularity phenomenon depending on the step of the group (for stratified groups) and on the order of the considered operator.
     \end{abstract}

\maketitle

\section{Introduction}

In this paper we are interested in the well-posedness of the following Cauchy problem: 
\begin{align}\label{CP}
\begin{cases}
\partial_t^2 u(t,x)+a(t)\mathcal R u(t,x)=0,\quad(t,x)\in[0,T]\times G,\\
u(0,x)=u_0(x),\quad x\in G,\\
\partial_t u(0,x)=u_1(x),\quad x\in G,
\end{cases}
\end{align}
for the time-dependent propagation speed $a=a(t)\geq 0$.

In the case of $G=\Rn$ and $\mathcal R=-\Delta$, the equation \eqref{CP} is the usual wave equation with the time-dependent propagation speed and its well-posedness results for H\"older regular functions $a$ have been obtained by Colombini, de Giorgi and Spagnolo in their seminal paper \cite{CDG1979}. Moreover, it has been shown by Colombini and Spagnolo in \cite{CS1982}  and by Colombini, Jannelli and Spagnolo in \cite{Colombini-Jannelli-Spagnolo:Annals-low-reg} that even in the case of $G=\R$ and $\mathcal R=-\frac{d^2}{dx^2}$ the Cauchy problem \eqref{CP} does not have to be well-posed in $\mathcal C^\infty$ if $a\in \mathcal C^\infty$ is not strictly positive or if it is in the H\"older class $a\in \mathcal C^\alpha$ for $0<\alpha<1$.

In this paper we obtain new results for the following situations:

\begin{itemize}
\item[(i)] $G=\mathbb H^n$ is the Heisenberg group and $\mathcal R$ is the positive Kohn-Laplacian on $G$.
\item[(ii)] $G$ is a stratified Lie group and $\mathcal R$ is a (positive) sub-Laplacian on $G$.
\item[(iii)] $G$ is a graded Lie group in the sense of Folland and Stein \cite{FS} and $\mathcal R$ is any positive Rockland operator on $G$, i.e. any positive left-invariant homogeneous hypoelliptic differential operator on $G$.
\end{itemize} 
In fact, our results are for the latter case (iii), the former two cases (i) and (ii) being its special cases. In particular, already in the cases of $G$ being the Euclidean space $\Rn$, the Heisenberg group $\mathbb H^n$, or any stratified Lie group, the case (iii) above allows one to consider $\mathcal R$ to be an operator of {\em any order}, as long as it is a positive left- (or right-) invariant homogeneous hypoelliptic differential operator. In the case of $\Rn$ these cases of so-called $p$-evolution equations have been studied in e.g. \cite{CHR08a,CHR08b, CC10}, however, for more restrictive conditions on $a(t)$ than those considered in this paper.

For $a(t)\equiv 1$ and $G$ being the Heisenberg group $\mathbb H^n$ with $\mathcal R$ being the positive sub-Laplacian, the wave equation \eqref{CP} was studied by M\"uller and Stein \cite{Muller-Stein:Lp-wave-Heis} and Nachman \cite{Nachman:wave-Heisenberg-CPDE-1982}.
Other noncommutative settings with $a(t)\equiv 1$ have been analysed as well, see e.g. Helgason \cite{Helgason:wave-eqns-hom-spaces-1984}.
For $G$ being a compact Lie group and $-\mathcal R$ any H\"ormander's sum of squares on $G$ the problem \eqref{CP} was studied in \cite{GR}, and so the results of the present paper provide a nilpotent counterpart of the results there.

Apart from an independent interest of the subelliptic setting of stratified or graded Lie groups, these settings are the model cases for many corresponding problems for general partial differential operators on manifolds in view of the celebrated lifting theorem of Rothschild and Stein \cite{Rothschild-Stein:AM-1976}.

From the point of view of the time-dependent coefficient $a(t)$, we aim at carrying out the comprehensive analysis, thus distinguishing between the following four cases:
\begin{enumerate}[\text{$\quad$ Case} 1:]
\item $a\in\mathcal {\rm Lip}([0,T])$, $a(t)\geq a_0>0$;
\item $a\in\mathcal C^\alpha([0,T])$, $0<\alpha<1$, $a(t)\geq a_0 >0$;
\item $a\in\mathcal C^l([0,T])$, $l \geq 2$, $a(t)\geq 0$;
\item $a\in\mathcal C^\alpha([0,T])$, with $0<\alpha<2$, $a(t)\geq0$.
\end{enumerate}
The first case is the simplest situation while in the forth case we have an irregular coefficient that is allowed to be zero at some points. The second and third situations are `intermediate' cases, in the sense that we have either the regularity or the strict positivity. 
We distinguish between these cases because the results and methods of proofs are rather different.

We note that if the operator $\mathcal R$ is not elliptic, the local approach to the Cauchy problem \eqref{CP} is problematic since the equation is only {\em weakly hyperbolic} already in Case 1 above. Consequently, since the equation \eqref{CP} in local coordinates is the space-dependent variable multiplicities problem, very little is known about its well-posedness. In this direction, only very special results for some second order operators are available, see e.g. Nishitani \cite{Nishitani:BSM-1983} or Melrose \cite{Melrose:wave-subelliptic-1986}.
Non-Lipschitz coefficients have been also much analysed, see e.g. Colombini and M\'etivier \cite{CM} or Colombini and Lerner \cite{CL}. In addition to already mentioned restrictions for the well-posedness, see also Colombini and M\'etivier \cite{CM-systems} for a recent overview from the point of view of systems.

In the case of $\Rn$ and $-\mathcal R$ being the Laplacian, the regularity of $a$ less than H\"older such as discontinuous or measure-valued $a$ have been considered in \cite{Garetto-Ruzhansky:ARMA}. However, such low regularity requires very different methods, and this problem for the general Cauchy problem \eqref{CP} will be considered elsewhere.

Wave equations with time dependent coefficients for general densely defined operators with discrete spectrum acting in Hilbert spaces have been considered in \cite{RT-ARMA}. However, that setting is different from the present one since the spectrum in our situation is continuous.

To formulate our results, let us briefly introduce some necessary notation following, for example, Folland and Stein \cite{FS}.
Let $G$ be a {\em graded Lie group}, i.e. a connected simply connected Lie group such that its
Lie algebra $\mathfrak g$ has  a vector space decomposition
\begin{equation}\label{EQ:graded}
\mathfrak g = \oplus_{j=1}^\infty V_j,
\end{equation} 
such that 
all but finitely many of the $V_j$'s are $\{0\}$ and 
$[V_i,V_j]\subset V_{i+j}$. A special case analysed in detail by Folland \cite{F75} is of {\em stratified Lie groups} when the first stratum $V_1$ generates $\mathfrak g$ as an algebra, see also Folland and Stein \cite{FS-CPAM}. A typical example of such Lie group is the Heisenberg group. In general, graded Lie groups are necessarily homogeneous and nilpotent. Moreover, any graded Lie group can be viewed as some $\Rn$ with a polynomial group law. We can also refer to \cite[Section 3.1]{FR2016} for a detailed discussion of graded Lie groups and their properties.

Let $\mathcal R$ be a positive Rockland operator on $G$, that is, a positive (in the operator sense) left-invariant differential operator which is homogeneous of degree $\nu>0$ and which satisfies the so-called Rockland condition. This means that for each representation $\pi\in\widehat G$, except for the trivial one, the operator $\pi(\mathcal R)$ is injective on the space of smooth vectors $\mathcal H^\infty_\pi$, i.e.
\begin{align}\label{Rockland}
\forall v\in\mathcal H^\infty_\pi \quad \pi(\mathcal R)v=0\implies v=0.
\end{align}
Alternative characterisations of such operators have been considered by Rockland \cite{Rockland} and Beals \cite{Beals-Rockland}, until the definitive result of Helffer and Nourrigat \cite{HN-79} saying that {\em Rockland operators are precisely the left-invariant homogeneous hypoelliptic differential operators on $G$.} The existence of Rockland operators on general nilpotent Lie groups characterises precisely the class of graded Lie groups \cite{Miller:80, tER:97}. An example of 
a positive Rockland operator is the positive sub-Laplacian on a stratified Lie group: if $G$ is a stratified Lie group and $\{X_1,\dots,X_k\}$ is a basis for the first stratum of its Lie algebra, then  the positive sub-Laplacian 
\[
\slp=-\sum_{j=1}^{k}X_j^2
\]
is a positive Rockland operator. Moreover, for any $m\in\mathbb N$, the operator
$$
\mathcal R=(-1)^{m}\sum_{j=1}^{k}X_j^{2m}
$$
is a positive Rockland operator on the stratified Lie group $G$.
More generally, for any graded Lie group $G\sim\Rn$, if $X_1,\ldots,X_n$ is the basis of its Lie algebra $\mathfrak g$ with dilation weights $\nu_1,\ldots,\nu_n$, i.e. with
\begin{equation}\label{EQ:dils}
D_r X_j=r^{\nu_j} X_j,\quad j=1,\ldots,n,\; r>0,
\end{equation} 
where $D_r$ are dilations on $\mathfrak g$, then
the operator
$$
\mathcal R=\sum_{j=1}^n (-1)^{\frac{\nu_0}{\nu_j}} a_j X_j^{2\frac{\nu_0}{\nu_j}},\quad a_j>0,
$$
is a Rockland operator of homogeneous degree $2\nu_0$, if $\nu_0$ is any common multiple of $\nu_1,\ldots,\nu_n$. We refer to \cite[Section 4.1.2]{FR2016} for other examples and a detailed discussion of Rockland operators and graded Lie groups. In the case of $\Rn$, all elliptic homogeneous differential operators with constant coefficients are Rockland operators.

To formulate our results we will need two scales of spaces, namely, Sobolev and Gevrey spaces, adapted to the setting of graded Lie groups. Thus, let $G$ be a graded Lie group and let $\mathcal R$ be a positive Rockland operator of homogeneous degree $\nu$.
For any real number $s\in\R$, the Sobolev space $H^s_\mathcal R(G)$ is the subspace of $\mathcal S'(G)$ obtained as the completion of the Schwartz space $\mathcal S(G)$ with respect to the Sobolev norm
\begin{equation}\label{EQ:dSob}
\|f\|_{H^s_\mathcal R(G)}:=\|(I+\mathcal R)^{\frac{s}{\nu}}f\|_{L^2(G)}.
\end{equation} 
For stratified Lie groups such spaces and their properties have been extensively analysed by Folland in \cite{F75} and on general graded Lie groups they have been investigated in \cite{FR:Sobolev,FR2016}. In particular, these spaces do not depend on a particular choice of the Rockland operartor $\mathcal R$ used in the definition \eqref{EQ:dSob}, see \cite[Theorem 4.4.20]{FR2016}).
These spaces perfectly suit Case 1 described above but already in the Euclidian case, with the elliptic Laplace operator instead of the hypoelliptic Rockland operator in the wave equation \eqref{CP}, if the coefficient $a(t)$ is not Lipschitz regular or may become zero, the Gevrey spaces appear naturally (see e.g. Bronshtein \cite{Bronshtein:TMMO-1980})
since we can not expect anymore the well-posedness in $\mathcal C^\infty(G)$ or $\mathcal D'(G)$. Indeed, Colombini and Spagnolo  exhibited a concrete example in \cite{CS1982} of a Cauchy problem for the time-dependent wave equation on $\R$ with smooth $a\geq 0$ which is not well-posed in  $\mathcal C^\infty(\R)$ or $\mathcal D'(\R)$.  

Thus, given $s\geq 1$, we define the Gevrey type space
\begin{align}\label{G}
\mathcal G^s_{\mathcal R}(G):=\{f\in\mathcal{C}^\infty(G)\,|\,\exists A>0\,:\,\|e^{A\mathcal R^{\frac{1}{2s}}}f\|_{L^2(G)}<\infty\}.
\end{align}
These spaces provide for a subelliptic version of the usual Gevrey spaces. For example,
for $G=\mathbb H^n$ being the Heisenberg group with the basis $X_1,\ldots,X_{2n}$ of the first stratum, it was shown in \cite{FRT20..} that $f\in \mathcal G^s_{\mathcal R}(\mathbb H^n)$
if and only if  there exist two constants $B,C>0$ such that for every $\alpha\in\N^{2n}_0$ the following inequality holds
\begin{equation}\label{gevrey}
\|\partial^\alpha f\|_{L^2(\mathbb H^n)}\leq C B^{|\alpha|}(\alpha!)^s,
\end{equation}
where $\partial^\alpha=Y_1\dots Y_{|\alpha|}$, with $Y_j\in\{X_1,\dots,X_{2n}\}$ for every $j=1,\dots,|\alpha|$ and $\sum_{Y_j=X_k}1=\alpha_k$ for every $k=1,\dots,2n$.

Gevrey spaces \eqref{G} and the corresponding spaces of ultradistributions have been considered on compact Lie groups and on compact manifolds in \cite{DR} and in \cite{DR16}, respectively.

By an argument similar to that in \cite{FRT20..} 
for the sub-Laplacian or in \cite[Theorem 2.4]{DR} for elliptic operators, it can be shown that if $\mathcal R$ is a positive Rockland operator of homogeneous degree $\nu$, then $f\in \mathcal G^s_{\mathcal R}(G)$ if and only if there exist constants $B,C>0$ such that for every  $k\in\mathbb N\cup\{0\}$ we have
\begin{equation}\label{EQ:Gevchar}
 \|\mathcal R^k f\|_{L^2(G)}\leq C B^{\nu k} ((\nu k)!)^s.
\end{equation} 
Since Sobolev spaces do not depend on a particular choice of the Rockland operator used in their definition, the characterisation \eqref{EQ:Gevchar} of the Gevrey spaces implies that the same is true for $\mathcal G^s_{\mathcal R}(G)$. 

Thus, we may drop the subscript $\mathcal R$ in $H^s_{\mathcal R}$ and $\mathcal G^s_{\mathcal R}$ but we may also keep using it to refer to the norms that we may be using.

\medskip
Let us now formulate the main theorem of our paper, where we consider the following four cases:
\begin{enumerate}[\textbf{$\qquad$ Case }\bf 1:]
\item $a\in\mathcal {\rm Lip}([0,T])$, $a(t)\geq a_0>0$;
\item $a\in\mathcal C^\alpha([0,T])$, $0<\alpha<1$, $a(t)\geq a_0 >0$;
\item $a\in\mathcal C^l([0,T])$, $l \geq 2$, $a(t)\geq 0$;
\item $a\in\mathcal C^\alpha([0,T])$, with $0<\alpha<2$, $a(t)\geq0$.
\end{enumerate}

These are the four cases to which we refer repeatedly throughout this paper.

\begin{theorem}\label{THM:main}
Let $G$ be a graded Lie group and let $\mathcal R$ be a positive Rockland operator of homogeneous degree $\nu$. 
Let $T>0$.
Then the following holds, referring respectively to Cases 1-4 above:
\begin{enumerate}[\text{Case} 1:]
\item Given $s\in\R$, if the initial Cauchy data $(u_0,u_1)$ are in $H_\mathcal R^{s+\frac{\nu}{2}}(G)\times H_\mathcal R^s(G)$, then there exists a unique solution of \eqref{CP} in the space $\mathcal C([0,T],H_\mathcal R^{s+\frac{\nu}{2}}(G))\cap\mathcal C^1([0,T],H_\mathcal R^s(G))$, satisfying the following inequality for all values of $t\in [0,T]$:
\begin{align}\label{inequality case 1}
\|u(t,\cdot)\|^2_{H_{\mathcal R}^{s+\frac{\nu}{2}}}+\|\partial_t u(t,\cdot)\|^2_{H_{\mathcal R}^{s}}\leq C(\|u_0\|^2_{H_{\mathcal R}^{s+\frac{\nu}{2}}}+\|u_1\|^2_{H_{\mathcal R}^{s}});
\end{align}

\item If the initial Cauchy data $(u_0, u_1)$ are in $\mathcal G^s_\mathcal R(G)\times \mathcal G^s_\mathcal R(G)$, then there exists a unique solution of \eqref{CP} in $\mathcal C^2([0,T],\mathcal G_\mathcal R^s(G))$, provided that
\[
1\leq s <1 +\frac{\alpha}{1-\alpha};
\]

\item If the initial Cauchy data $(u_0, u_1)$ are in $\mathcal G^s_\mathcal R(G)\times \mathcal G^s_\mathcal R(G)$, then there exists a unique solution of \eqref{CP} in $\mathcal C^2([0,T],\mathcal G_\mathcal R^s(G))$, provided that
\[
1\leq s <1 +\frac{l}{2};
\]

\item If the initial Cauchy data $(u_0, u_1)$ are in $\mathcal G^s_\mathcal R(G)\times \mathcal G^s_\mathcal R(G)$ then there exists a unique solution of \eqref{CP} in $\mathcal C^2([0,T],\mathcal G_\mathcal R^s(G))$, provided that
\[
1\leq s <1 +\alpha. 
\]
\end{enumerate}
\end{theorem}

As it will follow from the proof, in Cases 2 and 4, one can take the equalities $s=1 +\frac{\alpha}{1-\alpha}$ and $s=1 +\alpha$, respectively, provided that $T>0$ is small enough.
We refer to \cite{GR2012, GR2013} concerning the sharpness of the above Gevrey indices in the case of $G=\Rn$ and $\mathcal R=-\Delta$, and for further relevant references for that case.

\medskip
Let us formulate a corollary from Theorem \ref{THM:main} showing the local loss of regularity for the Cauchy problem \eqref{CP}. We recall that any graded Lie group $G$ can be identified, for example through the exponential mapping, with the Euclidean space $\Rn$ where $n$ is the topological dimension of $G$. Then, if $\nu_1,\ldots,\nu_n$ are the dilation weights on $G$ as in \eqref{EQ:dils}, for any $s\in\R$ we have the local Sobolev embedding theorems:
\begin{equation}\label{EQ:Sobemb}
H^{s/\nu_1}_{loc}(\Rn)\subset H^s_{\mathcal R, loc}(G)\subset H^{s/\nu_n}_{loc}(\Rn),
\end{equation} 
see \cite[Theorem 4.4.24]{FR2016}. If $G$ is a stratified Lie group, we have $\nu_1=1$ and $\nu_n$ is the step of $G$, i.e. the number of steps in the stratification of its Lie algebra. In other words, if $G$ is a stratified Lie group of step $r$ and $H^s(G)$ is the Sobolev space defined using (any) sub-Laplacian on $G$, then the embeddings \eqref{EQ:Sobemb} are reduced to
\begin{equation}\label{EQ:Sobembs}
H^{s}_{loc}(\Rn)\subset H^s_{loc}(G)\subset H^{s/r}_{loc}(\Rn).
\end{equation} 
These embeddings are sharp, see Folland \cite{F75}. Consequently, using the characterisation \eqref{EQ:Gevchar} of $\mathcal G^s_{\mathcal R}(G)$, we also obtain the embeddings 
\begin{equation}\label{EQ:Gevemb}
\mathcal G^{\nu_1 s}_{loc}(\Rn)\subset \mathcal G^s_{\mathcal R, loc}(G)\subset \mathcal G^{s\nu_n}_{loc}(\Rn),
\end{equation} 
where the space $\mathcal G^{\sigma}_{loc}(\Rn)$ is the usual Euclidean Gevrey space, namely, the space of all smooth functions $f\in \mathcal C^\infty(\Rn)$ such that for every compact set $K\subset\Rn$ there exist two constants $B,C>0$ such that for every $\alpha$ we have
\begin{equation}\label{gevreyRn}
|\partial^\alpha f(x)|\leq C B^{|\alpha|}(\alpha!)^\sigma \quad\textrm{ for all } x\in K.
\end{equation}
Consequently, if $G$ is a stratified Lie group of step $r$ we have the embeddings
\begin{equation}\label{EQ:Gevembs}
\mathcal G^{s}_{loc}(\Rn)\subset \mathcal G^s_{\mathcal R, loc}(G)\subset \mathcal G^{s r}_{loc}(\Rn).
\end{equation} 

Consequently, using these embeddings, we obtain the following local in space well-posedness result using the usual Euclidean Gevrey spaces. Here we may also assume that the Cauchy data are compactly supported due to the finite propagation speed of singularities. To emphasise the appearing phenomenon of local loss of Euclidean regularity we formulate it in the simplified setting of stratified Lie groups, with topological identification $G\sim\Rn$.

\begin{corollary}\label{COR:main}
Let $G\sim \Rn$ be a stratified Lie group of step $r$ and let $\mathcal R$ be a positive Rockland operator of homogeneous degree $\nu$ (for example, $\mathcal R$ can be a positive sub-Laplacian in which case we have $\nu=2$). 
Assume that the Cauchy data $(u_0,u_1)$ are compactly supported.
Then the following holds, referring respectively to Cases 1-4 above:
\begin{enumerate}[\text{Case} 1:]
\item Given $s\in\R$, if $(u_0,u_1)$ are  in  $H^{s+\frac{\nu}{2}}(\Rn)\times H^s(\Rn)$, then there exists a unique solution of \eqref{CP} in $\mathcal C([0,T],H^{(s+\frac{\nu}{2})/r}(\Rn))\cap\mathcal C^1([0,T],H^{s/r}(\Rn))$, satisfying the following inequality for all values of $t\in [0,T]$:
\begin{align}\label{inequality case 1}
\|u(t,\cdot)\|^2_{H^{(s+\frac{\nu}{2})/r}}+\|\partial_t u(t,\cdot)\|^2_{H^{s/r}}\leq C(\|u_0\|^2_{H^{s+\frac{\nu}{2}}}+\|u_1\|^2_{H^{s}});
\end{align}

\item If $(u_0, u_1)$ are in $\mathcal G^s(\Rn)\times \mathcal G^s(\Rn)$, then there exists a unique solution of \eqref{CP} in $\mathcal C^2([0,T],\mathcal G^{sr}(\Rn))$, provided that
\[
1< s <1 +\frac{\alpha}{1-\alpha};
\]

\item If $(u_0, u_1)$ are in $\mathcal G^s(\Rn)\times \mathcal G^s(\Rn)$, then there exists a unique solution of \eqref{CP} in $\mathcal C^2([0,T],\mathcal G^{sr}(\Rn))$, provided that
\[
1< s <1 +\frac{l}{2};
\]

\item If $(u_0, u_1)$ are in $\mathcal G^s(\Rn)\times \mathcal G^s(\Rn)$ then there exists a unique solution of \eqref{CP} in $\mathcal C^2([0,T],\mathcal G^{sr}(\Rn))$, provided that
\[
1< s <1 +\alpha. 
\]
\end{enumerate}
\end{corollary}
The statements in Cases 2-4 for $s=1$ are not so interesting, with the analytic well-posedness known in these case anyway, see Bony and Shapira \cite{Bony-Shapira:analytic-IM-1972}.

For $G=\Rn$ and $\mathcal R$ being the Laplacian, we have $r=1$ and there is no loss of regularity in any of the Cases 1-4, when the results are known from
\cite{CDG1979,CK2002,GR2012,GR2013,KS2006}.

However, already on the Heisenberg group with step $r=2$, we observe the local loss of regularity in Euclidean Sobolev and Gevrey spaces in all statements of Cases 1-4 in Corollary \ref{COR:main}.

We also note that using local Sobolev and Gevrey embeddings \eqref{EQ:Sobemb} and \eqref{EQ:Gevemb}, it is easy to formulate an extension of Corollary \ref{COR:main} to general graded Lie groups.

\section{Preliminaries on graded Lie groups and Rockland operators}
\label{SEC:prelim}

In this section we recall some preliminaries and fix the notation concerning the Fourier analysis on graded Lie groups. We refer to \cite{FS} and to \cite[Chapter 5]{FR2016} for further details.

Thus, a connected and simply connected Lie group $G$ is called graded when its Lie algebra is graded in the sense of the decomposition \eqref{EQ:graded}.

A Lie algebra $\mathfrak g$ is stratified if it is graded and if its first stratum $V_1$ generates $\mathfrak g$ as an algebra. Thus, in this case every element of the Lie algebra can be written as a linear combination of elements in  $V_1$ and their iterated commutators.
A Lie group is stratified when it is connected, simply connected and its Lie algebra is stratified.

Furthermore, if there are $r$ non zero $V_j$'s in the vector space decomposition \eqref{EQ:graded}, then the group (respectively the algebra) is said to be stratified of step $r$.

From the definition of a stratified Lie algebra, it follows that, assuming that $V_1$ has dimension $k$, any basis $\{X_1,\dots,X_k\}$ for $V_1$ forms a H\"ormander system, see \cite{H1967}, and we can consider its associated sub-Laplacian operator that is also a positive Rockland operator:
 \begin{align}\label{slp}
 \slp:=-\big(X_1^2+\dots+X_k^2\big).
 \end{align}
 
\begin{example}[The Heisenberg group]
A classical example of a graded (stratified) Lie group is the Heisenberg group $\mathbb H^n$ that might be seen as the manifold $\R^{2n+1}$ endowed with the group law
\[
(x,y,t)(x',y',t'):=(x+x',y+y',t+t'+\frac{1}{2}(x\cdot y'-x'\cdot y)),
\]
 where $(x,y,t),\, (x',y',t')\in\R^n\times\R^n\times\R\sim\h^n$.
The \textit{Heisenberg Lie algebra} $\mathfrak{h}^n$ associated with the Heisenberg group is the space of all the left-invariant vector fields of $\mathbb{H}^n$ and it admits the following canonical basis:
\begin{align}\notag
&X_j=\partial_{x_j}-\frac{y_j}{2}\partial_t,\\\notag
&Y_j=\partial_{y_j}+\frac{x_j}{2}\partial_t,\\\notag
&T=\partial_t.
\end{align}
The former vector fields satisfy the canonical commutation relations
\[
[X_j,Y_j]=T,\quad\forall\,j=1,\dots,n,
\] 
and all the other possible commutators are zero. Therefore, the Heisenberg group is a graded (stratified) Lie group of step $2$, whose Lie algebra admits the vector space decomposition
\begin{align*}
\mathfrak h^n= V_1\oplus V_2,
\end{align*}
where
\begin{align*}
V_1=\sum_{j=1}^n \R X_j \oplus \R Y_j\quad\text{ and }\quad V_2=\R T.
\end{align*}
Hypoellipticity and other questions on the Heisenberg group have a long history, see e.g. Taylor \cite{T1986}, Folland \cite{Fol}, or Thangavelu \cite{thangavelu}, and many references therein.
\end{example}

From now on, we consider $G$ to be a graded Lie group, even if some of the following definitions and remarks hold in a more general setting.

Let $\pi$ be a representation of $G$ on the separable Hilbert space $\mathcal H_\pi$. A vector $v\in\mathcal H_\pi$ is said to be smooth or of type $\mathcal C^\infty$ if the function 
\[
G\ni x\mapsto \pi(x)v\in\mathcal H_\pi
\]
is of class $\mathcal C^\infty$. The space of all smooth vectors of a representation $\pi$ is denoted by $\mathcal H_\pi^\infty$.
Let $\mathfrak g $ be the Lie algebra of $G$ and let $\pi$ be a strongly continuous representation of $G$ on a Hilbert space $\mathcal H_\pi$. For every $X\in\mathfrak g$ and $v\in\mathcal H_\pi^\infty$ we define 
\[
d\pi(X)v:=\lim_{t\rightarrow 0}\frac{1}{t}\Big(\pi\big(\exp_G(tX)\big)v-v\Big).
\] 
Then $d\pi$ is a representation of $\mathfrak g$ on $\mathcal H_\pi^\infty$ (see e.g. \cite[Proposition 1.7.3]{FR2016}) called the infinitesimal representation associated to $\pi$. By abuse of notation, we will often still denote it by $\pi$, therefore, for any $X\in\mathfrak g$, we write $\pi(X)$ meaning $d\pi(X)$.

Any left-invariant differential operator $T$ on $G$, according to the Poincar\'e-Birkhoff-Witt theorem, can be written in a unique way as a finite sum
\begin{equation}\label{EQ:invs}
T=\sum_{|\alpha|\leq M}c_\alpha X^\alpha,
\end{equation} 
where all but finitely many of the coefficients $c_\alpha\in\C$ are zero and $X^\alpha=X_1\dots X_{|\alpha|}$, with $X_j\in\mathfrak g$. This allows one to look at any left-invariant differential operator $T$ on $G$ as an element of the universal enveloping algebra $\mathfrak U(\mathfrak g)$ of the Lie algebra of $G$. Therefore, the family of infinitesimal representations $\big\{\pi(T),\,\pi\in\widehat G \big\}$ yields a field of operators that turns to be the symbol associated with the operator $T$. 

Let $\pi\in\widehat G$ and let $\mathcal R$ be a positive Rockland operator of homogeneous degree $\nu>0$, then using formula \eqref{EQ:invs}, the infinitesimal representation of $\mathcal R$ associated to $\pi$ is 
\[
\pi(\mathcal R)=\sum_{[\alpha]=\nu}c_\alpha\pi({X})^\alpha,
\]
where $\pi({X})^\alpha=\pi({X}^\alpha)=\pi(X_1^{\alpha_1}\cdots X_n^{\alpha_n})$ and $[\alpha]=\nu_1\alpha_1+\cdots+\nu_n\alpha_n$ is the homogeneous degree of the multiindex $\alpha$, with $X_j$ being homogeneous of degree $\nu_j$.

The operator $\mathcal R$ and its infinitesimal representations $\pi(\mathcal R)$ are densely defined on $\mathcal D(G)\subset L^2(G)$ and $\mathcal H_\pi^\infty\subset \mathcal H_\pi$, respectively, see e.g. \cite[Proposition 4.1.15]{FR2016}. We denote by $\mathcal R_2$ the self-adjoint extension of $\mathcal R$ on $L^2(G)$ and we keep the same notation $\pi(\mathcal R)$ for the self-adjoint extensions on $\mathcal H_\pi$ of the infinitesimal representations. Recalling the spectral theorem for unbounded operators \cite[Theorem VIII.6]{RS1980}, we can consider the spectral measures $E$ and $E_\pi$ corresponding to $\mathcal R_2$ and $\pi(\mathcal R)$, so that we have
\begin{align}\notag
\mathcal R_2=\int_\R \lambda dE(\lambda)\quad\text{and}\quad\pi(\mathcal R)=\int_\R\lambda dE_\pi(\lambda).
\end{align} 
Furthermore,  for any $f\in L^2(G)$ we have
\begin{align}\label{functionalcalculus}
\mathcal F\Big(\phi\big(\mathcal R\big)f\Big)(\pi)=\phi\big(\pi(\mathcal R)\big)\widehat f(\pi),
\end{align}
for any measurable bounded function $\phi$ on $\R$, see e.g. \cite[Corollary 4.1.16]{FR2016}.
The infinitesimal representations $\pi(\mathcal R)$ of a positive Rockland operator are also positive, due to the relations between their spectral measures. In particular, Hulanicki, Jenkins and Ludwig showed in  \cite{HJL1985} that the spectrum of $\pi(\mathcal R)$, with $\pi\in\widehat G\setminus\{1\}$, is discrete and lies in $(0,\infty)$. This implies that we can choose an orthonormal basis for $\mathcal H_\pi$ such that the infinite matrix associated to the self-adjoint operator $\pi(\mathcal R)$ has the form
\begin{align}\label{EQ:R-spec}
\pi(\mathcal R)=\begin{pmatrix}
\pi^2_1 & 0      & \dots    & \dots \\
0      & \pi^2_2 & 0    & \dots \\
\vdots&   0       & \ddots  &          \\
\vdots& \vdots  &             & \ddots
\end{pmatrix},
\end{align}
where $\pi_j$ are strictly positive real numbers and $\pi\in\widehat G\setminus\{1\}$.

\section{Parameter dependent energy estimates}
\label{SEC:ODE}

In this section we prove certain energy estimates for second order ordinary differential equations with explicit dependence on parameters. This will be crucial in the proof of Theorem \ref{THM:main} where the parameters will correspond to the spectral decomposition \eqref{EQ:R-spec} of the infinitesimal representations of the Rockland operators.

Results of the following type have been of use in different estimates related to weakly hyperbolic partial differential equations, such as \cite{CK2002} and \cite{GR}. However, in those papers the conclusions rely on more general results, see \cite{GR2012}. We partly follow the argument in \cite{GR} based on a standard reduction to a first order system. Consequently, we carry out different types of arguments depending on assumptions in each of the cases, altogether allowing us to formulate the precise dependence on parameters for ordinary differential equations corresponding to the propagation coefficient $a(t)$ as in Cases 1-4 of Theorem \ref{THM:main}, to which we refer in the following statement. 

\begin{proposition}\label{lemma}
Let $\beta>0$ be a positive constant and let $a(t)$ be a function that behaves according to Cases $1,2,3$ and $4$ in Theorem \ref{THM:main}. Let $T>0$. Consider the following Cauchy problem:
\begin{align}\label{ODE}
\begin{cases}
v''(t)+\beta^2 a(t)v(t)=0\quad\text{with }t\in(0,T],\\
v(0)=v_0\in\C,\\
 v'(0)=v_1\in\C.
\end{cases}
\end{align}  
Then the following holds: 
\begin{enumerate}[\text{Case} 1:]
\item There exists a positive constant $C>0$  such that for all $t\in [0,T]$ we have
\[
\beta^2|v(t)|^2+|v'(t)|^2\leq C(\beta^2|v_0|^2+|v_1|^2).
\]
\item There exist two positive constants $C,K>0$ such that for all $t\in [0,T]$ we have
\begin{equation}\label{EQ:case2}
\beta^2|v(t)|^2+|v'(t)|^2\leq Ce^{Kt \beta^{\frac{1}{s}}}(\beta^2|v_0|^2+|v_1|^2),
\end{equation} 
for any
$1\leq s < 1+\frac{\alpha}{1-\alpha}$. Moreover, there exists a constant $k>0$ such that
for any $\beta_0\geq 1$ the estimate \eqref{EQ:case2} holds for $K= k\beta_0^{1-\alpha-\frac1s}$ for all $\beta\geq \beta_0$.
\item There exist two positive constants $C,K>0$ such that for all $t\in [0,T]$ we have
\[
\beta^2|v(t)|^2+|v'(t)|^2\leq C(1+\beta^\frac{l}{\sigma})e^{K\beta^\frac{1}{\sigma}}\big(\beta^2|v_0|^2+|v_1|^2\big),
\]
with $\sigma= 1+\frac{l}{2}$. 
\item There exist two positive constants $C,K>0$ such that 
\begin{equation}\label{EQ:case4}
\beta^2|v(t)|^2+|v'(t)|^2\leq C(1+\beta^\frac{\alpha}{\alpha+1})e^{Kt\beta^\frac{1}{s}}(\beta^2|v_0|^2+|v_1|^2),
\end{equation} 
for any $1\leq s<1+\alpha$.
Moreover, there exists a constant $k>0$ such that
for any $\beta_0\geq 1$ the estimate \eqref{EQ:case4} holds for $K= k\beta_0^{\frac{1}{1+\alpha}-\frac1s}$ for all $\beta\geq \beta_0$.
\end{enumerate}
The constants $C$ in the above inequalities may depend on $T$ but not on $\beta$.
\end{proposition}

\begin{proof}
First we reduce the problem \eqref{ODE} to a first order system. In order to do this we rewrite it in a standard way as a matrix-valued equation. Thus we define the column vectors
\[
V(t):=\begin{pmatrix}i\beta v(t)\\ \partial_t v(t)\end{pmatrix},\quad V_0:=\begin{pmatrix}i\beta v_0\\ v_1\end{pmatrix},
\]
and the matrix 
\begin{align*}
A(t):=\begin{pmatrix}
0 & 1\\
a(t) & 0
\end{pmatrix},
\end{align*}
that allow us to reformulate the second order system \eqref{ODE} as the first order system
\begin{align}\label{first}
\begin{cases}
V_t(t)=i\beta A(t)V(t),\\
V(0)=V_0.
\end{cases}
\end{align}
We will now treat each case separately.

\subsection{Case 1: $\mathbf{a\in\mathbf{Lip}([0,T])}$, $\mathbf{a(t)\geq a_0>0}$.}

This is the simplest case that can be treated by a classical argument.
We observe that the eigenvalues of our matrix $A(t)$ are given by $\pm \sqrt{a(t)}$. The symmetriser $S$ of $A$, i.e. the matrix such that 
\[
SA-A^*S=0,
\]
is given by
\[
S(t)=\begin{pmatrix}2 a(t) & 0\\ 0 & 2\end{pmatrix}.
\]
Thus we define the energy as 
\[
E(t):=\Big(S(t)V(t),V(t)\Big),
\] 
and we want to estimate its variations in time. A straightforward calculation yields the following inequality that will help us to get such estimate:
\begin{equation}\label{en1}
2 |V|^2 \min_{t\in[0,T]} \{a(t),1\}\leq E(t)\leq 2|V|^2\max_{t\in[0,T]}\{a(t),1\}. 
\end{equation}
In particular, in this case the continuity of $a(t)$ ensures the existence of two strictly positive constants $a_0$ and $a_1$ such that
\[
a_0=\min_{t\in[0,T]} a(t)\quad\text{and}\quad a_1=\max_{t\in[0,T]}a(t).
\]
Thus setting $c_0:=2 \min \{a_0,1\}$ and $c_1:=2\max\{a_1,1\}$, the inequality \eqref{en1} becomes
\begin{align}\label{en2}
c_0|V(t)|^2\leq E(t)\leq c_1 |V(t)|^2.
\end{align}
A straightforward calculation, together with \eqref{en2}, gives the following estimate:
\begin{align}\notag
E_t(t)&=\big( S_t(t)V(t),V(t)\big)+\big(S(t)V_t(t),V(t)\big)+\big(S(t)V(t),V_t(t)\big)=\\\notag
&=\big( S_t(t)V(t),V(t)\big) +i\beta\big( S(t)A(t)V(t),V(t)\big)-i\beta\big( S(t)V(t),A(t)V(t)\big)=\\\notag
&=\big( S_t(t)V(t),V(t)\big)+i\beta\Big( \big(S(t)A(t)-A^*(t)S(t)\big)V(t),V(t)\Big)=\\
&=\big( S_t(t)V(t),V(t)\big)\leq \|S_t(t)\| |V(t)|^2,\label{der}
\end{align}
thus setting $c':=c_0^{-1}\sup_{t\in [0,T]}\|S_t(t)\|$, we get from \eqref{der} using \eqref{en2} that
\begin{align}\label{gron}
E_t(t)\leq c' E(t).
\end{align}
Applying the Gronwall lemma to \eqref{gron}, we deduce that there exists a constant $c>0$ independent of $t\in[0,T]$ such that 
\begin{align}\label{gron2}
E(t)\leq c E(0).
\end{align}
Therefore, putting together \eqref{gron2} and \eqref{en2} we obtain
\begin{align}\notag
c_0 |V(t)|^2\leq E(t)\leq c E(0)\leq c c_1 |V(0)|^2.
\end{align} 
We can then rephrase this, asserting that there exists a constant $C>0$ independent of $t$ such that $|V(t)|^2\leq C |V(0)|^2$. Then we write this inequality going back to the definition of $V(t)$, yielding
\[
\beta^2|v(t)|^2+|\partial_t v(t)|^2\leq C\big(\beta^2 |v_0|^2 + |v_1|^2\big),
\]
as required.

\subsection{Case 2: $\mathbf{a\in\mathcal C^\alpha([0,T])}$, with  $\mathbf{ 0<\alpha<1,\,\,a(t)\geq a_0>0}$.}
Here we follow the method developed by Colombini and Kinoshita \cite{CK2002} for $n=1$ and subsequently extended \cite{GR2012} for any $n\in\mathbb N$. We look for solutions of the form 
\begin{align}\label{sol}
V(t)=e^{-\rho(t)\beta^{\frac{1}{s}}}(\det H(t))^{-1}H(t)W(t),
\end{align}
where
\begin{itemize}
\item $s\in\R$  depends on $\alpha$ as will be determined in the argument;
\item the function $\rho=\rho(t)\in\mathcal C^1([0,T])$ is real-valued and will be chosen later;
\item $W(t)$ is the energy;
\item $H(t)$ is the matrix defined by
\[
H(t):=\begin{pmatrix}1 & 1\\ \lambda^\epsilon_1(t) & \lambda^\epsilon_2(t)\end{pmatrix},
\]
where for all $\epsilon>0$, $\lambda^\epsilon_1(t)$ and $\lambda^\epsilon_2(t)$ are regularisations of the eigenvalues of the matrix $A(t)$ of the form
\begin{align*}
&\lambda^\epsilon_{1}(t):=(-\sqrt{a}*\varphi_\epsilon)(t),\\\notag
&\lambda^\epsilon_{2}(t):=(+\sqrt{a}*\varphi_\epsilon)(t),
\end{align*}
with $\{\varphi_\epsilon(t)\}_{\epsilon>0}$ being a family of cut-off functions defined starting from a non-negative function $\varphi\in\mathcal C^\infty_c(\R)$, with $\int_\R\varphi=1$, by setting $\varphi_\epsilon(t):=\frac{1}{\epsilon}\varphi\big(\frac{t}{\epsilon}\big)$. By construction, it follows that $\lambda^\epsilon_{1},\lambda^\epsilon_{2}\in\mathcal C^\infty([0,T])$.

\end{itemize}
Furthermore, we can easily check, using the H\"older regularity of $a(t)$ of order $\alpha$ and, therefore, of $\sqrt{a(t)}$ of the same order $\alpha$, the following inequalities: 
 \begin{equation}\label{strait}
\det H(t)=\lambda^\epsilon_2(t)-\lambda^\epsilon_1(t)\geq 2\sqrt{a(0)}\geq 2\sqrt{a_0},
\end{equation}
and for all $t\in[0,T]$ and $\epsilon>0$ there exist two constants $c_1,c_2>0$ such that
\begin{align}\notag
|\lambda^\epsilon_1(t)+\sqrt{a(t)}|\leq c_1\epsilon^\alpha,\\\label{difference}
|\lambda^\epsilon_2(t)-\sqrt{a(t)}|\leq c_2 \epsilon^\alpha,
\end{align}
uniformly in $t$ and $\epsilon$.
Now we substitute our suggested solution \eqref{sol} in \eqref{first} yielding
\begin{align*}
&-\rho'(t)\beta^\frac{1}{s}e^{-\rho(t)\beta^\frac{1}{s}}\frac{H(t)W(t)}{\det H(t)}+e^{-\rho(t)\beta^\frac{1}{s}}\frac{H_t(t)W(t)}{\det H(t)}+e^{-\rho(t)\beta^\frac{1}{s}}\frac{H(t)W_t(t)}{\det H(t)}+\\
&-e^{-\rho(t)\beta^\frac{1}{s}}(\det H)_t(t)\frac{H(t)W(t)}{\big(\det H(t)\big)^2}=i\beta A(t)e^{-\rho(t)\beta^\frac{1}{s}}\frac{H(t)W(t)}{\det H(t)}.
\end{align*}
Multiplying both sides of this equality  by $e^{\rho(t)\beta^{\frac{1}{s}}}\det H(t)H^{-1}(t)$ we get
\begin{align}\notag
W_t(t)&=\rho'(t)\beta^{\frac{1}{s}}W(t)-H^{-1}(t)H_t(t)W(t)+(\det H)_t(t)\big(\det H(t)\big)^{-1}W(t)+\\
&+i\beta H^{-1}(t)A(t)H(t)W(t).\label{W}
\end{align}
This leads to the estimate
\begin{align*}
\frac{d}{dt}|W(t)|^2&=\big(W_t(t),W(t)\big)+\big(W(t),W_t(t)\big)=2\Rep\big(W_t(t),W(t)\big)=\\
&=2\Big(\rho'(t)\beta^{\frac{1}{s}}|W(t)|^2-\Rep\big(H^{-1}(t)H_t(t)W(t),W(t)\big)+\\
&+\big(\det H(t)\big)^{-1}(\det H)_t(t)|W(t)|^2+\beta\Imp\big(H^{-1}(t)A(t)H(t)W(t),W(t)\big)\Big).
\end{align*}
We observe that
\begin{align*}
&2\Imp\big(H^{-1}AHW,W\big)=\big(H^{-1}AHW,W\big)-\overline{\big(H^{-1}AHW,W\big)}=\\
&=\big(H^{-1}AHW,W\big)-\big(W,H^{-1}AHW\big)=\big(H^{-1}AHW,W\big)-\big((H^{-1}AH)^*W,W\big)=\\
&=\Big(\big(H^{-1}AH-(H^{-1}AH)^*\big)W,W\Big)\leq\|H^{-1}AH-(H^{-1}AH)^*\|\|W\|^2.
\end{align*}
 Thus we obtain
 \begin{align}\notag
\frac{d}{dt}|W(t)|^2\leq \Big(2\rho'(t)\beta^{\frac{1}{s}}+2\|H^{-1}(t)H_t(t)\|+2\big|\big(\det H(t)\big)^{-1}(\det H)_t(t)\big|+\\\label{enW}
+\beta \|H^{-1}AH-(H^{-1}AH)^*\| \Big)|W(t)|^2.
 \end{align}
To proceed we need to estimate the following quantities:
\begin{enumerate}[I)]
\item $\|H^{-1}(t)H_t(t)\|$;
\item $\big|\big(\det H(t)\big)^{-1}(\det H)_t(t)\big|$;
\item $\|H^{-1}AH-(H^{-1}AH)^*\|$.
\end{enumerate}
In \cite{GR2012} and \cite{CK2002}, the authors determine estimates for similar functions in a more general setting, i.e. starting from an equation of arbitrary order $m$. In this particular case, we can proceed by straightforward calculations without relying on the mentioned works. 

We deal with these three terms as follows:
\begin{enumerate}[I)]
\item Since $H^{-1}(t)=\frac{1}{\lambda_2^\epsilon-\lambda_1^\epsilon}\begin{pmatrix}\lambda_2^\epsilon &-1\\-\lambda_1^\epsilon & 1
\end{pmatrix}$ and $H_t(t)=\begin{pmatrix}0 & 0\\\partial_t\lambda_1^\epsilon & \partial_t\lambda_2^\epsilon
\end{pmatrix}$, it follows that the entries of the matrix $H^{-1}H_t$ are given by the functions $\frac{\partial_t \lambda_j^\epsilon}{\lambda_2^\epsilon-\lambda_1^\epsilon}$. We have, for example for $\lambda_2$,
\begin{align}\notag
\partial_t\lambda_2^\epsilon(t)&=\sqrt a * \partial_t \varphi_\epsilon(t)=\frac{1}{\epsilon^2}\sqrt a *\varphi'\Big(\frac{t}{\epsilon}\Big)=\frac{1}{\epsilon}\int\sqrt{a(t-\rho\epsilon)}\varphi'(\rho)d\rho=\\\label{derL}
&=\frac{1}{\epsilon}\int\big(\sqrt{a(t-\rho\epsilon)}-\sqrt{a(t)}\big)\varphi'(\rho)d\rho+\frac{1}{\epsilon}\sqrt{a(t)}\int\varphi'(\rho)d\rho\leq k\epsilon^{\alpha -1},
\end{align}
where we are using the H\"older continuity of $\sqrt a$ for the first term and the fact that the second term is zero, since $\int \varphi'=0 $.
Combining the inequalities \eqref{strait} and \eqref{derL}, we get for a suitable positive constant $k_1$ that
\[
\|H^{-1}(t)H_t(t)\|\leq k_1\epsilon^{\alpha-1}.
\]

\item First we can estimate
\begin{align*}
\big|\big(\det H(t)\big)^{-1}(\det H)_t(t)\big|= \frac{\partial_t\lambda_2^\epsilon-\partial_t\lambda_1^\epsilon}{\lambda_2^\epsilon-\lambda_1^\epsilon}=\frac{2\partial_t\lambda_2^\epsilon}{\lambda_2^\epsilon-\lambda_1^\epsilon}\leq\frac{2k\epsilon^{\alpha-1}}{2\sqrt{a_0}},
\end{align*}
therefore,
\[
\big|\big(\det H(t)\big)^{-1}(\det H)_t(t)\big|\leq k_2 \epsilon^{\alpha-1},
\]
for a constant $k_2>0$.
\item Also in this case, we write explicitly the matrix we are interested in, that is
\begin{equation}\notag
H^{-1}AH-\big(H^{-1}AH\big)^*=\begin{pmatrix}0&\frac{-2a(t)+(\lambda_1^\epsilon)^2+(\lambda_2^\epsilon)^2}{\lambda_1^\epsilon-\lambda_2^\epsilon}\\ \frac{2a(t)-\big((\lambda_1^\epsilon)^2+(\lambda_2^\epsilon)^2\big)}{\lambda_1^\epsilon-\lambda_2^\epsilon}&0\end{pmatrix}.
\end{equation}
Observing that, by definition, $(\lambda_1^\epsilon)^2=(\lambda_2^\epsilon)^2$, and recalling inequality \eqref{strait}, to get the desired norm estimate, it is enough to consider the function $|a(t)-(\lambda_2^\epsilon)^2|$. A straightforward calculation, using inequality \eqref{difference}, shows that
\begin{align*}
|a(t)-(\lambda_2^\epsilon)^2|&=|\big(\sqrt{a(t)}-\lambda_2^\epsilon\big)\big(\sqrt{a(t)}+\lambda_2^\epsilon\big)|\leq\\
&\leq c_2\epsilon^\alpha\Big(\sqrt{a(t)}+\int\sqrt{a(t-s)}\varphi_\epsilon(s)ds\Big)=\\&=c_2\epsilon^\alpha+\int \big(\sqrt{a(t)} +\sqrt{a(t-s\epsilon)}\big)\varphi(s)ds\leq 2c_2\|\sqrt a\|_{L^\infty} \epsilon^\alpha.
\end{align*}
It follows that
\[
\|H^{-1}AH-(H^{-1}AH)^*\|\leq k_3 \epsilon^\alpha.
\]
\end{enumerate}
Going back to \eqref{enW}, combining it with estimates I), II) and III), we get an estimate for the derivative of the energy, that is
\begin{align}\label{enW2}
\frac{d}{dt}|W(t)|^2\leq \Big(2\rho'(t)\beta^{\frac{1}{s}}+2k_1\epsilon^{\alpha-1}+2k_2 \epsilon^{\alpha-1}+ k_3 \beta \epsilon^\alpha \Big)|W(t)|^2.
\end{align}
At this point we choose $\epsilon=\frac{1}{\beta}$, observing that we can always consider $\beta$ large enough, say $\beta>1$, in order to have a small $\epsilon\in(0,1]$. 
Indeed, for $\beta\leq\beta_0$ for some fixed $\beta_0>0$, a modification of the argument below gives estimate \eqref{EQ:case2} with constants depending only on $\beta_0$ and $T$. So we may assume that $\beta>\beta_0$ for $\beta_0$ to be specified.
We define also $\rho(t):=\rho(0)-Kt$ for some $K>0$ to be specified. 
Substituting this in \eqref{enW2} we get for a suitable constant $k>0$ that 
\begin{align*}
\frac{d}{dt}|W(t)|^2\leq \Big(2\rho'(t)\beta^{\frac{1}{s}}+2k\beta^{1-\alpha}\Big)|W(t)|^2= \big(-2K+2k\beta^{1-\alpha-\frac1s} \big)\beta^{\frac{1}{s}} |W(t)|^2.
\end{align*}
If we have
\[
\frac{1}{s}>1-\alpha \quad \iff\quad s<1+\frac{\alpha}{1-\alpha},
\]
and then also 
\begin{equation}\label{EQ:Kchoice}
K:= k\beta_0^{1-\alpha-\frac1s}\geq k\beta^{1-\alpha-\frac1s},
\end{equation} 
then for all $t\in[0,T]$ we have
\[
\frac{d}{dt}|W(t)|^2\leq 0.
\]
This monotonicity of the energy $W(t)$ yields the following boundedness for the solution vector $V(t)$:
\begin{align}\notag
|V(t)|&=e^{-\rho(t)\beta^{\frac{1}{s}}}\big(\det H(t)\big)^{-1}\|H(t)\||W(t)|\leq\\
&\leq e^{-\rho(t)\beta^{\frac{1}{s}}}\big(\det H(t)\big)^{-1}\|H(t)\||W(0)|=e^{K t\beta^{\frac{1}{s}}}\frac{\det H(0)}{\det H(t)}\frac{\|H(t)\|}{\|H(0)\|}|V(0)|.\label{V}
\end{align}
Note that, according to property \eqref{strait}, the function $\big(\det H\big)^{-1}(t)$ is bounded. Furthermore, the  behaviour of the convolution and the definition of $H(t)$ guarantee the existence of suitable constants $c,c'>0$ such that $\|H(t)\|\leq c$ and $\|H^{-1}(0)\|\leq c'$. Therefore, there exists a constant $C>0$ such that
\begin{align*}
|V(t)|\leq C e^{Kt\beta^{\frac{1}{s}}}|V(0)|,
\end{align*}
that means, by definition of $V(t)$, that
\begin{align*}
\beta^2|v(t)|^2+|v_t(t)|\leq C e^{Kt\beta^{\frac{1}{s}}}\big(\beta^2|v_0|^2+|v_1|^2\big),
\end{align*}
proving the statement of Case 2.

\subsection{Case 3: $\mathbf{a\in\mathcal C^l([0,T])}$, with $\mathbf{l\geq 2,\; a(t)\geq 0}$}
In this case we extend the technique developed for Case 1. First we perturb the symmetriser of the matrix $A(t)$. This is done considering the so-called quasi-symmetriser of $A(t)$, the idea introduced for such problems by D'Ancona and Spagnolo in \cite{DS1998}.

Consider the quasi-symmetriser of $A(t)$, that is, a family of coercive, Hermitian matrices of the form
\begin{align*}
Q_\epsilon^{(2)}(t):=S(t)+2\epsilon^2\begin{pmatrix}1 & 0\\ 0 & 0\end{pmatrix}=\begin{pmatrix}2 a(t) & 0\\ 0 & 2\end{pmatrix}+2\epsilon^2\begin{pmatrix}1 & 0\\ 0 & 0\end{pmatrix},
\end{align*}
for all $\epsilon\in(0,1]$, and such that $\big(Q_\epsilon^{(2)}A-A^{*} Q_\epsilon^{(2)}\big)$ goes to zero as $\epsilon$ goes to zero. The associated perturbed energy is given by 
\begin{align*}
E_\epsilon(t):=\Big(Q_\epsilon^{(2)}V(t),V(t)\Big).
\end{align*}
We proceed estimating the energy, calculating its derivatives in time, so that
\begin{align}\notag
&\frac{d}{dt}E_\epsilon(t)= \Big(\frac{d}{dt}Q_\epsilon^{(2)}(t)V(t),V(t)\Big)+(Q_\epsilon^{(2)}V_t(t),V(t))+(Q_\epsilon^{(2)}V(t),V_t(t))=\\ \label{energyest}
&=\Big(\frac{d}{dt}Q_\epsilon^{(2)}(t)V(t),V(t)\Big)+i\beta\Big(\big(Q_\epsilon^{(2)}A(t)-A^*(t)Q_\epsilon^{(2)}\big)V(t),V(t)\Big).
\end{align}
To estimate the second term in the right hand side, we set 
\[
V(t)=\begin{pmatrix}
i\beta v(t)\\
\partial_t v
\end{pmatrix}=:\begin{pmatrix}
v_1\\
v_2
\end{pmatrix}.
\]
Algebraic calculations give
\begin{align}\notag
Q_\epsilon^{(2)}(t)A(t)-A^*(t)Q_\epsilon^{(2)}(t)=2\epsilon^2 
\begin{pmatrix}
0 & 1\\
-1 & 0
\end{pmatrix},
\end{align}
therefore
\begin{align}\notag
&i\Big(\big(Q_\epsilon^{(2)}(t)A(t)-A^*(t)Q_\epsilon^{(2)}(t)\big)V(t),V(t)\Big)\leq 2\epsilon^2\int 2\Imp(v_2\overline{v_1})dt \leq 2\epsilon \int 2|\epsilon v_1||v_2|dt\leq\\\notag
&\leq 2\epsilon\int (\epsilon^2|v_1|^2+|v_2|^2)dt\leq 2\epsilon\int\big(\epsilon^2+a(t)\big)|v_1|^2+|v_2|^2dt= 2\epsilon \big(Q_\epsilon^{(2)}(t)V(t),V(t)\big).
\end{align}
Using this estimate in \eqref{energyest}, we get
\begin{align}\notag
&\frac{d}{dt}E_\epsilon(t)= \Big(\frac{d}{dt}Q_\epsilon^{(2)}(t)V(t),V(t)\Big)+i\beta\Big(\big(Q_\epsilon^{(2)}A(t)-A^*(t)Q_\epsilon^{(2)}\big)V(t),V(t)\Big)\leq\\ \notag
&\leq \Big(\frac{d}{dt}Q_\epsilon^{(2)}(t)V(t),V(t)\Big) +2\beta\epsilon E_\epsilon(t)\\
&= \Bigg[\frac{\Big(\frac{d}{dt}Q_\epsilon^{(2)}(t)V(t),V(t)\Big)}{\big(Q_\epsilon^{(2)}(t)V(t),V(t)\big)}+ 2 \beta \epsilon\Bigg]E_\epsilon(t).
\end{align}
In order to apply the Gronwall lemma, we first estimate the integral
\begin{align}\label{lemma2}
\int_0^T\frac{\big(\frac{d}{dt}Q_\epsilon^{(2)}(t)V(t),V(t)\big)}{\big(Q_\epsilon^{(2)}(t)V(t),V(t)\big)}dt.
\end{align}
Let us recall that from the definition of the quasi-symmetriser, it follows that
\begin{align}\label{qs}
\Big( Q_\epsilon^{(2)}V,V\Big)=2\int\big(\big(a(t)+\epsilon^2\big)\beta^2|v|^2+|\partial_t v|^2\big)dt.
\end{align} 
Thus, setting $c_1:=\max\big(1, 2(\|a\|_{L^\infty}+\epsilon^2)\big)$, we obtain a bound from above for \eqref{qs}, that is
\[
\Big( Q_\epsilon^{(2)}V,V\Big)\leq c_1 |V|^2.
\]
Observing that $\epsilon^2c_1^{-1}\leq 1$ and $\epsilon^2c_1^{-1}\leq c_1$ for small enough $\epsilon$, we can also deduce an inequality from below of the form
\[
\epsilon^2c_1^{-1}|V|^2\leq \Big( Q_\epsilon^{(2)}V,V\Big).
\]
Hence, there exists a constant $c_1\geq 1$ such that for $t\in[0,T]$ we have
\begin{align}\label{property1}
c_1^{-1}\epsilon^2|V(t)|^2\leq \big(Q_\epsilon^{(2)}(t)V(t),V(t)\big)\leq c_1 |V(t)|^2.
\end{align}
The lower bound, together with \cite[Lemma 2]{GR2013} (see \cite[Lemma 2]{KS2006} for a detailed proof), allows us to estimate the integral \eqref{lemma2} as follows
\begin{align}\notag
&\int_0^T\frac{\big(\frac{d}{dt}Q_\epsilon^{(2)}(t)V(t),V(t)\big)}{\big(Q_\epsilon^{(2)}(t)V(t),V(t)\big)}dt \leq
\int_0^T\frac{\big(\frac{d}{dt}Q_\epsilon^{(2)}(t)V(t),V(t)\big)}{\big(Q_\epsilon^{(2)}(t)V(t),V(t)\big)^{1-\frac{1}{l}}(Q_\epsilon^{(2)}(t)V(t),V(t)\big)^\frac{1}{l}}dt\leq\\ \notag
&\leq c_1^\frac{1}{l}\epsilon^{-\frac{2}{l}} \int_0^T\frac{\big(\frac{d}{dt}Q_\epsilon^{(2)}(t)V(t),V(t)\big)}{\big(Q_\epsilon^{(2)}(t)V(t),V(t)\big)^{1-\frac{1}{l}}|V(t)|^\frac{2}{l}}dt \leq c_1^\frac{1}{l}\epsilon^{-\frac{2}{l}} c_T\|Q_\epsilon^{(2)}\|^{\frac{1}{l}}_{\mathcal C^l([0,T])}\leq c_3\epsilon^{-\frac{2}{l}}.
\end{align}
Thus, by the Gronwall lemma and the estimates for the quasi-symmetriser just derived, we obtain
\[
E_\epsilon(t)\leq E_\epsilon(0)e^{c_3\epsilon^{-\frac{2}{l}}+2\beta\epsilon T}.
\]
Combining the latter inequality with \eqref{property1} we obtain
\[
c_1^{-1}\epsilon^2 |V(t)|^2\leq E_\epsilon(t)\leq E_\epsilon(0)e^{c_T(\epsilon^{-\frac{2}{l}}+\beta\epsilon)} \leq c_1|V(0)|^2 e^{c_T(\epsilon^{-\frac{2}{l}}+\beta\epsilon)}.
\]
We choose $\epsilon$ such that $\epsilon^{-\frac{2}{l}}=\beta\epsilon$, thus $\epsilon=\beta^{-\frac{l}{2+l}}$ and $\epsilon\beta=\beta^{\frac{2}{2+l}}$. 
We can assume $\beta$ is large enough with a remark for small $\beta$ similar to Case 2.
Setting $\sigma=1+\frac{l}{2}$, for a suitable constant $K\in\C$ it follows that
\[
|V(t)|^2\leq C \beta^{\frac{l}{\sigma}}e^{K\beta^{\frac{1}{\sigma}}}|V(0)|^2.
\]
This means that
\[
\beta^2|v(t)|^2+|v'(t)|^2\leq C \beta^\frac{l}{\sigma}e^{K\beta^{\frac{1}{\sigma}}}\big(\beta^2 |v_0|^2+|v_1|^2\big),
\]
as required.

\subsection{\textbf{Case 4: $\mathbf{a\in\mathcal C^\alpha([0,T])}$, with $\mathbf{0<\alpha<2}$, $\mathbf{a(t)\geq0}$.}}

In this last case we extend the proof of Case 2. However, under these assumptions the roots of the matrix $A(t)$, that are, $\pm\sqrt{a(t)}$ might coincide, and hence they are H\"older of order $\frac{\alpha}{2}$ instead of $\alpha$. In order to adapt this proof to the one for Case 2 we will assume without loss of generality that $a\in\mathcal C^{2\alpha}([0,T])$ with $0<\alpha<1$, so that  $\sqrt{a}\in\mathcal C^\alpha([0,T])$. 

Following the argument developed for Case 2, we look again for solutions of the form
\begin{align*}
V(t)=e^{-\rho(t)\beta^{\frac{1}{s}}}\big(\det  H(t)\big)^{-1}H(t)W(t),
\end{align*}
with the real-valued function $\rho(t)$, the exponent $s$ and the energy $W(t)$ to be chosen later, while $H(t)$ is the matrix given by
\begin{align*} 
H(t)=\begin{pmatrix}
1 & 1\\
\lambda^\epsilon_{1,\alpha}(t) & \lambda^\epsilon_{2,\alpha}(t)
\end{pmatrix},
\end{align*}
where the regularised eigenvalues of $A(t)$, $\lambda^\epsilon_{1,\alpha}(t)$ and $\lambda^\epsilon_{2,\alpha}(t)$ differ from the ones defined in the previous case in the following way
\begin{align*}
&\lambda^\epsilon_{1,\alpha}(t):=(-\sqrt{a}*\varphi_\epsilon)(t)+\epsilon^\alpha,\\
&\lambda^\epsilon_{2,\alpha}(t):=(+\sqrt{a}*\varphi_\epsilon)(t)+2\epsilon^\alpha.
\end{align*}
Arguing as in Case 2, we can easily see that the smooth functions $\lambda^\epsilon_{1}(t)$ and $\lambda^\epsilon_{2}(t)$ satisfy uniformly in $t$ and $\epsilon$ the following inequalities
\begin{itemize}
\item $\det  H(t)=\lambda^\epsilon_{2,\alpha}(t)-\lambda^\epsilon_{1,\alpha}(t)\geq c_1\epsilon^\alpha$;
\item $|\lambda^\epsilon_{1,\alpha}(t)+\sqrt {a(t)}|\leq c_2 \epsilon^\alpha$;
\item $|\lambda^\epsilon_{2,\alpha}(t)-\sqrt{a(t)}|\leq c_3 \epsilon^\alpha$.
\end{itemize}
We now look for the energy estimates. In order to do this, recalling the calculations done before \eqref{W} and \eqref{enW}, we obtain
\begin{align}\notag
&\frac{d}{dt}\big|W(t)\big|^2=2\Rep\big(W_t(t),W(t)\big)\leq \Big(2\rho'(t)\beta^{\frac{1}{s}}+2\|H^{-1}(t)H_t(t)\|+\\&+2\big|\big(\det H(t)\big)^{-1}\det H_t(t)\big|+\beta \|H^{-1}AH-(H^{-1}AH)^*\| \Big)|W(t)|^2.\label{W4}
\end{align}
The same arguments as in Case 2 allow us to get the following bounds
\begin{enumerate}[I)]
\item $\|H^{-1}(t)H_t(t)\|\leq k_1\epsilon^{-1}$;
\item $\big|\big(\det H(t)\big)^{-1}\det H_t(t)\big|\leq k_2 \epsilon^{-1}$;
\item $\|H^{-1}AH-(H^{-1}AH)^*\|\leq k_3 \epsilon^\alpha$.
\end{enumerate}
Combining \eqref{W4} with I), II) and III) we obtain 
\begin{align*}
\frac{d}{dt}|W(t)|^2\leq \Big(2\rho'(t)\beta^{\frac{1}{s}}+2k_1\epsilon^{-1}+2k_2 \epsilon^{-1}+ k_3 \beta \epsilon^\alpha \Big)|W(t)|^2.
\end{align*}
We choose $\epsilon^{-1}=\beta\epsilon^{\alpha}$ which yields $\epsilon=\beta^{-\frac{1}{\alpha+1}}$. Thus, setting $\gamma:=\frac{1}{\alpha+1}$, we obtain for a constant $c>0$ the estimate
\begin{align*}
\frac{d}{dt}|W(t)|^2\leq \Big(2\rho'(t)\beta^{\frac{1}{s}}+2c \beta^\gamma \Big)|W(t)|^2.
\end{align*}
We take $\rho(t):=\rho(0)-Kt$ with $K>0$ to be chosen later. 
Considering 
\begin{align*}
\frac{1}{s}>\gamma\quad\iff\quad s< 1 +\alpha,
\end{align*}
we get
\[
\frac{d}{dt}|W(t)|^2\leq (-2K+2c\beta^{\gamma-\frac1s}) \beta^{\frac{1}{s}} |W(t)|^2\leq 0,
\]
provided that $\beta$ is large enough.
Similarly to Case 2, we then get
\begin{align}\label{V2}
|V(t)|\leq e^{Kt\beta^{\frac{1}{s}}}\big(\det H(t)\big)^{-1}\det H(0)\|H(t)\|\big(\|H(0)\|\big)^{-1}|V(0)|.
\end{align}
Since 
\[
\big(\det H(t)\big)^{-1}\det H(0)\|H(t)\|\big(\|H(0)\|\big)^{-1}\leq c\epsilon^{-\alpha}=c\beta^{\frac{\alpha}{\alpha+1}},
\]
the inequality \eqref{V2} becomes 
\begin{align*}
|V(t)|\leq c \beta^{\frac{\alpha}{\alpha+1}}e^{Kt\beta^{\frac{1}{s}}}|V(0)|,
\end{align*}
which means 
\begin{align*}
\beta^2|v(t)|^2+|v'(t)|^2\leq c\beta^{\frac{\alpha}{\alpha+1}}e^{Kt\beta^{\frac{1}{s}}}\big(\beta^2|v_0|^2+|v_1|^2\big).
\end{align*}
Combining this with a remark for small $\beta$ similar to Case 2 yields the result.
Thus Proposition \ref{lemma} is proved.
\end{proof}

\section{Proof of Theorem \ref{THM:main}}
\label{SEC:proof}

In this section we combine the things from Section \ref{SEC:prelim} and Section \ref{SEC:ODE} to prove Theorem \ref{THM:main}.
However, we will need one more ingredient, the Fourier transform on $G$, that we now briefly describe. 

Let $f\in L^1(G)$ and let $\pi\in\widehat G$. 
By a usual abuse of notation we will identify irreducible unitary representations with their equivalence classes.
The group Fourier transform of $f$ at $\pi$ is defined by
\[
\mathcal F_G f(\pi)\equiv\widehat f(\pi)\equiv\pi(f):=\int_G f(x)\pi(x)^*dx,
\] 
with integration against the biinvariant Haar measure on $G$.
This gives a linear mapping $\widehat f(\pi):\mathcal H_\pi\rightarrow\mathcal H_\pi$ that can be represented by an infinite matrix  once we choose a basis for the Hilbert space $\mathcal H_\pi$.
Consequently, we can write
\[
\mathcal F_G \big(\mathcal R f\big)(\pi)=\pi({\mathcal R})\widehat f(\pi).
\]
By Kirillov's orbit method (see e.g. \cite{CG90}), one can explicitly construct the Plancherel measure $\mu$ on the dual $\widehat G$. Therefore we can have the Fourier inversion formula. In addition,  the operator $\pi(f)=\widehat f (\pi)$ is Hilbert-Schmidt:
\[
\|\pi(f)\|^2_{\HS}={\rm Tr}\big(\pi(f)\pi(f)^*\big)<\infty,
\]
and the function $\widehat G\ni\pi \mapsto \|\pi(f)\|^2_{\HS}$ in integrable with respect to $\mu$. Furthermore, the Plancherel formula holds:
\begin{align}\label{plancherel}
\int_G|f(x)|^2dx=\int_{\widehat G}\|\pi(f)\|^2_{\HS}d\mu(\pi),
\end{align}
see e.g. \cite{CG90} or \cite{FR2016}.

\begin{proof}[Proof of Theorem \ref{THM:main}]
Our aim is to reduce the Cauchy problem \eqref{CP} to a form allowing us to apply Proposition \ref{lemma}. In order to do this, we take the group Fourier transform of \eqref{CP} with respect to $x\in G$ for all $\pi \in\widehat G $, that is,
\begin{align}\label{fourier}
\partial_t^2 \widehat u (t,\pi)+a(t)\pi(\mathcal R)\widehat u (t,\pi)=0.
\end{align}
Keeping in mind the form \eqref{EQ:R-spec} of the infinitesimal representation $\pi(\mathcal R)$ the equation \eqref{fourier} can be seen componentwise as an infinite system of equations of the form
\begin{align}\label{single}
\partial_t^2 \widehat u (t,\pi)_{m,k}+a(t)\pi_m^2\widehat u (t,\pi)_{m,k}=0,
\end{align}
where we are considering any $\pi\in\widehat G$, and any $m,k\in\N$.
The key point of the following argument is to decouple the system given by the matrix equation \eqref{fourier}. In order to do this, we fix an arbitrary representation $\pi$, and a general entry $(m,k)$ and we treat each equation given by \eqref{single} individually. Note that eventually $\widehat u(t,\pi)_{m,k}$ is a function only of $t$. Formally, recalling the notation used in Proposition \ref{lemma}, we write
\[
v(t):=\widehat u(t,\pi)_{m,k},\quad \beta^2:=\pi_m^2,
\]
and
\[
v_0:=\widehat u_0 (\pi)_{m,k},\quad v_1:=\widehat u_1(\pi)_{m,k}.
\]
Therefore, equation \eqref{single} becomes
\[
v''(t)+a(t)\beta^2 v(t)=0.
\]
We proceed discussing implications of Proposition \ref{lemma} separately in each case.

\medskip
\textbf{Case 1: $\mathbf{a\in\mathbf{Lip}([0,T])}$, $\mathbf{a(t)\geq a_0>0}$.}

Applying Proposition \ref{lemma}, we get that there exists a positive constant $C>0$  such that 
\[
\beta^2|v(t)|^2+|v'(t)|^2\leq C(\beta^2|v_0|^2+|v_1|^2),
\]
which is equivalent to
\begin{align}\label{single inequality}
|\pi_m\widehat u(t,\pi)_{m,k}|^2+|\widehat u'(t,\pi)_{m,k}|^2\leq C \big(|\pi_m\widehat u_0(\pi)_{m,k}|^2+|\widehat u_1(\pi)_{m,k}|^2\big).
\end{align}
This holds uniformly in $\pi\in\widehat G$ and $m, k\in\N$. We multiply the inequality \eqref{single inequality} by $\pi_m^{4s/\nu}$ yielding
\begin{align}\label{general single inequality}
|\pi_m^{1+\frac{2s}{\nu}}\widehat u(t,\pi)_{m,k}|^2+|\pi_m^{\frac{2s}{\nu}}\widehat u'(t,\pi)_{m,k}|^2\leq C \big(|\pi_m^{1+\frac{2s}{\nu}}\widehat u_0(\pi)_{m,k}|^2+|\pi_m^{\frac{2s}{\nu}}\widehat u_1(\pi)_{m,k}|^2\big).
\end{align}
Thus, recalling that for any Hilbert-Schmidt operator $A$ we have 
$$\|A\|^2_{\HS}=\sum_{m,k}|(A\varphi_m,\varphi_k)|^2$$ for any orthonormal basis $\{\varphi_1,\varphi_2,\dots\}$, we can consider the infinite sum over $m,k$ of the inequalities provided by \eqref{general single inequality}, to get
\begin{align}\label{HS inequality}
\|\pi(\mathcal R)^{\frac{1}{2}+\frac{s}{\nu}}\widehat u (t,\pi)\|_{\HS}^2+\|\pi(\mathcal R)^{\frac{s}{\nu}}&\partial_t \widehat u(t,\pi)\|_{\HS}^2\leq\\\notag
&\leq C\big(\|\pi(\mathcal R)^{\frac{1}{2}+\frac{s}{\nu}}\widehat u_0(\pi)\|_{\HS}^2+\|\pi(\mathcal R)^{\frac{s}{\nu}}\widehat u_1(\pi)\|_{\HS}^2\big).
\end{align}
We can now integrate both sides of \eqref{HS inequality} against the Plancherel measure $\mu$ on $\widehat G$, so that the Plancherel identity yields estimate \eqref{inequality case 1}.

\medskip
\textbf{Case 2: $\mathbf{a\in\mathcal C^\alpha([0,T])}$, with } $\mathbf{ 0<\alpha<1,\,\,a(t)\geq a_0>0}$.

The application of Proposition \ref{lemma} implies the existence of two positive constants $C,K>0$ such that for all $m,k\in\N$ and for every representation $\pi\in\widehat G$ we have
\begin{align}\label{lemma case 2}
|\pi_m\widehat u (t,\pi)_{m,k}|^2+|\widehat u' (t,\pi)_{m,k}|^2\leq Ce^{Kt\pi_m^{\frac{1}{s}}}(|\pi_m\widehat u_0 (\pi)_{m,k}|^2+|\widehat u_1 (\pi)_{m,k}|^2),
\end{align} 
where 
\[
s < 1+\frac{\alpha}{1-\alpha}.
\]
If the Cauchy data $(u_0, u_1)$ are in $\mathcal G^s_\mathcal R(G)\times \mathcal G^s_\mathcal R(G)$ then there exist two positive constants $A_0$ and $A_1$ such that
\begin{equation*}
\|e^{A_0\mathcal R^{\frac{1}{2s}}}u_0\|_{L^2}<\infty\quad\text{and}\quad\|e^{A_1\mathcal R^{\frac{1}{2s}}}u_1\|_{L^2}<\infty.
\end{equation*}
We note that we can restrict to consider $\pi_m$ large enough since the cut-off to bounded $\pi_m$ produces functions in any Gevrey spaces. Indeed, if a cut-off $\chi:\R\to\C$ has a compact support, then by the same energy estimate, since $\chi(\mathcal R)$ and $\mathcal R$ commute, the problem is reduced to the solution $\chi(\mathcal R) u(t,x)$ to the Cauchy problem
\begin{align}\label{CPc}
\begin{cases}
\partial_t^2 (\chi(\mathcal R) u(t,x))+a(t)\mathcal R (\chi(\mathcal R)u(t,x))=0,\quad(t,x)\in[0,T]\times G,\\
\chi(\mathcal R)u(0,x)=\chi(\mathcal R)u_0(x),\quad x\in G,\\
\partial_t(\chi(\mathcal R) u(0,x))=\chi(\mathcal R)u_1(x),\quad x\in G,
\end{cases}
\end{align}
with data $\chi(\mathcal R)u_0$ and $\chi(\mathcal R)u_1$, so it
is in any Gevrey class.

Take now $A=\min\{A_0,A_1\}$, then we can always assume $K$ in Case 2 of Proposition \ref{lemma} is small enough, so that we have some $B>0$ such that $KT=A-B$. Therefore, we can rewrite inequality \eqref{lemma case 2} as 
\begin{align}\label{lemma case 3}
e^{B\pi_m^{\frac{1}{2s}}}\big(|\pi_m\widehat u (t,\pi)_{m,k}|^2+|\widehat u' (t,\pi)_{m,k}|^2\big)\leq Ce^{A\pi_m^{\frac{1}{s}}}(|\pi_m\widehat u_0 (t)_{m,k}|^2+|\widehat u_1 (t)_{m,k}|^2).
\end{align}
Summing over $m,k$, integrating against the Plancherel measure of $\widehat G$ and applying the Plancherel identity, inequality \eqref{lemma case 3} becomes
\begin{align}\notag
\|e^{B\mathcal R^{\frac{1}{2s}}}u\|_{L^2(G)}&+\|e^{B\mathcal R^{\frac{1}{2s}}}\partial_tu\|_{L^2(G)}\leq
\|e^{B\mathcal R^{\frac{1}{2s}}}\mathcal R^{\frac{1}{2}}u\|_{L^2(G)}+\|e^{B\mathcal R^{\frac{1}{2s}}}\partial_tu\|_{L^2(G)}\leq\\\label{case 2 conclusion}
&\leq \|e^{B\mathcal R^{\frac{1}{2s}}}\mathcal R^{\frac{1}{2}}u_0\|_{L^2(G)}+\|e^{B\mathcal R^{\frac{1}{2s}}}u_1\|_{L^2(G)}.
\end{align}
If a function $f$ belongs to $\mathcal G^s_{\mathcal R}(G)$, then also $\mathcal R^\frac{1}{2}f$ is in $\mathcal G_\mathcal R^s(G)$. Therefore, from \eqref{case 2 conclusion} we get the desired well-posedness result.

\medskip
$\mathbf{Case\, 3: a\in\mathcal C^l([0,T]),\, \,with\,\, l\geq 2,\, \,a(t)\geq 0.}$

Similarly to the previous cases, the application of Proposition \ref{lemma} yields the existence of two positive constants $C,K>0$ such that
\begin{align}\notag
|\pi_m \widehat{u}(t,\pi)_{m,k}(t)|^2+|\widehat{u}'(t,\pi)_{m,k}|^2&\leq C\pi_m^{\frac{l}{\sigma}+2}e^{KT\pi_m^\frac{1}{\sigma}}| \widehat{u}_0(\pi)_{m,k}|^2+C\pi_m^\frac{l}{\sigma}e^{KT\pi_m^\frac{1}{\sigma}}|\widehat{u}_1(\pi)_{m,k}|^2\leq\\\notag
&\leq Ce^{K'\pi_m^\frac{1}{s}}| \widehat{u}_0(\pi)_{m.k}|^2+Ce^{K'\pi_m^\frac{1}{s}}|\widehat{u}_1(\pi)_{m,k}|^2,
\end{align}
with $1\leq s<\sigma=1+\frac{l}{2}$, for some $K'>0$ small enough. Proceeding as in Case 2, we obtain the desired inequality.

\medskip
\textbf{Case 4: $\mathbf{a\in\mathcal C^\alpha([0,T])}$, with $\mathbf{0<\alpha<2}$, $\mathbf{a(t)\geq0}$.}

In this last case, applying Proposition \ref{lemma} we have that there exist two positive constants $C,K>0$ such that 
\[
\pi_m^2|\widehat{u}(t.\pi)_{m.k}|^2+|\widehat{u}'(t,\pi)_{m,k}|^2\leq C\pi_m^\frac{1}{\alpha+1}e^{KT\pi_m^\frac{1}{s}}(\pi_m^2|\widehat{u}_0(\pi)_{m,k}|^2+|\widehat{u}_1(\pi)_{m,k}|^2),
\]
with $1\leq s< \alpha +1$. Arguing as above, the result follows.
\end{proof}


\begin{thebibliography}{9}

\bibitem{Beals-Rockland}
R.~Beals.
\newblock Op\'erateurs invariants hypoelliptiques sur un groupe de {L}ie
  nilpotent.
\newblock {\em S\'eminaire {G}oulaouic-{S}chwartz 1976/1977: \'Equations aux
  d\'eriv\'ees partielles et analyse fonctionnelle}, {E}xp. {N}o. 19:8pp, 1977.
  
  \bibitem{Bony-Shapira:analytic-IM-1972}
J.-M. Bony, P.~Schapira.
\newblock Existence et prolongement des solutions holomorphes des {\'e}quations
  aux d{\'e}riv{\'e}es partielles.
\newblock {\em Invent. Math.}, 17:95--105, 1972.

\bibitem{Bronshtein:TMMO-1980}
M.~D. Bron{\v{s}}te{\u\i}n.
\newblock The {C}auchy problem for hyperbolic operators with characteristics of
  variable multiplicity.
\newblock {\em Trudy Moskov. Mat. Obshch.}, 41:83--99, 1980.

\bibitem{CHR08a}
M. Cicognani, F. Hirosawa, M. Reissig.
{The Log-effect for p-evolution type models.} 
{\em J. Math. Soc. Japan}, 60:819--863, 2008.

\bibitem{CHR08b}
M. Cicognani, F. Hirosawa, M. Reissig. 
{Loss of regularity for p-evolution type models.} 
{\em J. Math. Anal. Appl.}, 347:35--58, 2008.

\bibitem{CC10}
M. Cicognani, F. Colombini.
{The Cauchy problem for p-evolution equations},
{\em Trans. Amer. Math. Soc.}, 362:4853--4869, 2010.

\bibitem{CDG1979} 
F. Colombini, E. De Giorgi, S. Spagnolo.
{Sur les \'equations hyperboliques avec des coefficients qui ne d\'ependent que du temps},
 {\em Ann. Scuola Norm. Sup. Pisa Cl. Sci.}, 6:511--559, 1979.
 
\bibitem{Colombini-Jannelli-Spagnolo:Annals-low-reg}
F.~Colombini, E.~Jannelli, S.~Spagnolo.
\newblock Nonuniqueness in hyperbolic {C}auchy problems.
\newblock {\em Ann. of Math. (2)}, 126(3):495--524, 1987. 
 
\bibitem{CK2002} F. Colombini, T. Kinoshita.
{On the Gevrey well-posedness of the Cauchy problem for weakly hyperbolic equations of higher order}, {\em J. Differential Equations}, 186:394--419, 2002.
 
 \bibitem{CL}
F. Colombini, N. Lerner.
Hyperbolic operators with non-Lipschitz coefficients. 
{\em Duke Math. J.},  77:657--698, 1995.
 
 \bibitem{CM}
 F. Colombini, G. M\'etivier.
 The Cauchy problem for wave equations with non Lipschitz coefficients; application to continuation of solutions of some nonlinear wave equations. 
 {\em Ann. Sci. \'Ec. Norm. Sup\'er.}, (4) 41:177--220, 2008.
 
 \bibitem{CM-systems}
 F. Colombini, G. M\'etivier.
Counterexamples to the well posedness of the Cauchy problem for hyperbolic systems.
{\em Anal. PDE}, 8:499--511, 2015. 
 
 \bibitem{CS1982} F. Colombini, S. Spagnolo.
 {An example of a weakly hyperbolic Cauchy problem not well posed in $\mathcal{C}^\infty$}, 
 {\em Acta Math.}, 148:243--253, 1982.


\bibitem{CG90} L. J. Corwin, F. P. Greenleaf.
\emph{Representations of nilpotent Lie groups and their applications}, 
Cambridge Studies in Advanced Mathematics, Cambridge University Press, Cambridge, ${18}$ (1990). Basic theory and examples.

\bibitem{DS1998} P. D'Ancona, S. Spagnolo.
{Quasi-symmetrization of hyperbolic systems and propagation of the analytic regularity}, 
{\em Bollettino della Unione Matematica Italiana}, 1 B:169--186, 1998.


\bibitem{DR}
D. Dasgupta, M. Ruzhansky.
{Gevrey functions and ultradistributions on compact Lie groups and homogeneous spaces},
{\em Bull. Sci. Math.}, 756--782, 2014.

\bibitem{DR16}
D. Dasgupta, M. Ruzhansky.
{Eigenfunction expansions of ultradifferentiable functions and ultradistributions}, 
{\em Trans. Amer. Math. Soc.}, 368:8481--8498, 2016.

\bibitem{FR:Sobolev}
V.~Fischer, M.~Ruzhansky.
\newblock {S}obolev spaces on graded groups.
\newblock {\em to appear in Ann. Inst. Fourier,
  https://arxiv.org/abs/1311.0192}, 2013.

\bibitem{FR2016} V. Fischer, M. Ruzhansky.
\emph{Quantization on nilpotent Lie groups}, Progress in Mathematics, vol. 314, Birkh\"auser, 
2016. (open access book)

\bibitem{FRT20..} V. Fischer, M. Ruzhansky, C. Taranto. 
\emph{On the sub-Laplacian Gevrey spaces},
preprint.

\bibitem{F75}
G.~B. Folland.
\newblock Subelliptic estimates and function spaces on nilpotent {L}ie groups.
\newblock {\em Ark. Mat.}, 13(2):161--207, 1975.

\bibitem{Fol} G. B. Folland.
\emph{Harmonic Analysis in Phase Space},
Princeton, Princeton University Press, 2nd ed.,
1989. 

\bibitem{FS-CPAM} G.~B.~Folland, E.~M.~Stein.
\newblock Estimates for the $\overline{\partial_{b}}$
complex and analysis on the Heisenberg group.
\newblock {\em Comm. Pure Appl. Math.}, 27:429--522, 1974.

\bibitem{FS}
G.~B. Folland, E.~M. Stein.
\newblock {\em Hardy spaces on homogeneous groups}, volume~28 of {\em
  Mathematical Notes}.
\newblock Princeton University Press, Princeton, N.J.; University of Tokyo
  Press, Tokyo, 1982.
  
\bibitem{GR2012}
C. Garetto, M. Ruzhansky.
{On the well-posedness of weakly hyperbolic equations with time-dependent coefficients},
{\em J. Differential Equations}, 253:1317--1340, 2012.

\bibitem{GR2013}
C. Garetto, M. Ruzhansky.
{Weakly hyperbolic equations with non-analytic coefficients and lower order terms}, 
{\em Math. Ann.}, 357:401--440, 2013.

\bibitem{GR}
C. Garetto, M. Ruzhansky.
{Wave equation for sum of squares on compact Lie groups},
{\em J. Differential Equations}, 258:4324--4347, 2015.

\bibitem{Garetto-Ruzhansky:ARMA}
C.~Garetto, M.~Ruzhansky.
\newblock Hyperbolic second order equations with non-regular time dependent
  coefficients.
\newblock {\em Arch. Rational Mech. Anal.}, 217:113--154, 2015.
  
 \bibitem{HN-79}
B.~Helffer, J.~Nourrigat.
\newblock Caracterisation des op\'erateurs hypoelliptiques homog\`enes
  invariants \`a gauche sur un groupe de {L}ie nilpotent gradu\'e.
\newblock {\em Comm. Partial Differential Equations}, 4(8):899--958, 1979. 
  
  \bibitem{Helgason:wave-eqns-hom-spaces-1984}
S.~Helgason.
\newblock Wave equations on homogeneous spaces.
\newblock In {\em Lie group representations, {III} ({C}ollege {P}ark, {M}d.,
  1982/1983)}, volume 1077 of {\em Lecture Notes in Math.}, pages 254--287.
  Springer, Berlin, 1984.

 
\bibitem{H1967} L. H\"ormander.
{Hypoelliptic second order differential equations},
{\em Acta Math.}, 119:147--171, 1967.

\bibitem{HJL1985} A. Hulanicki, J. W. Jenkins, J. Ludwig.
{Minimum eigenvalues for positive, Rockland operators},
{\em Proc. Amer. Math. Soc.}, 94:718--720, 1985.


\bibitem{KS2006} T. Kinoshita, S. Spagnolo.
{Hyperbolic equations with non-analytic coefficients},
{\em Math. Ann.}, 336:551--569, 2006.

\bibitem{Melrose:wave-subelliptic-1986}
R.~Melrose.
\newblock Propagation for the wave group of a positive subelliptic second-order
  differential operator.
\newblock In {\em Hyperbolic equations and related topics ({K}atata/{K}yoto,
  1984)}, pages 181--192. Academic Press, Boston, MA, 1986.
  
 \bibitem{Miller:80}
K.~G. Miller.
\newblock Parametrices for hypoelliptic operators on step two nilpotent {L}ie
  groups.
\newblock {\em Comm. Partial Differential Equations}, 5(11):1153--1184, 1980.

\bibitem{Muller-Stein:Lp-wave-Heis}
D.~M{{\"u}}ller, E.~M. Stein.
\newblock {$L^p$}-estimates for the wave equation on the {H}eisenberg group.
\newblock {\em Rev. Mat. Iberoamericana}, 15(2):297--334, 1999.

\bibitem{Nachman:wave-Heisenberg-CPDE-1982}
A.~I. Nachman.
\newblock The wave equation on the {H}eisenberg group.
\newblock {\em Comm. Partial Differential Equations}, 7(6):675--714, 1982.

\bibitem{Nishitani:BSM-1983}
T.~Nishitani.
\newblock Sur les {\'e}quations hyperboliques {\`a} coefficients
  h{\"o}ld{\'e}riens en {$t$} et de classe de {G}evrey en {$x$}.
\newblock {\em Bull. Sci. Math. (2)}, 107(2):113--138, 1983.

\bibitem{RS1980} M. Reed, B. Simon.
\emph{Methods of Modern Mathematical Physics, Vol. $1$: Functional Analysis, revised and enlarged edition},
Academic Press, 1980.

\bibitem{Rockland}
C.~Rockland.
\newblock Hypoellipticity on the {H}eisenberg group-representation-theoretic
  criteria.
\newblock {\em Trans. Amer. Math. Soc.}, 240:1--52, 1978.

\bibitem{Rothschild-Stein:AM-1976}
L.~P. Rothschild, E.~M. Stein.
\newblock Hypoelliptic differential operators and nilpotent groups.
\newblock {\em Acta Math.}, 137(3-4):247--320, 1976.

\bibitem{RT-ARMA}  M. Ruzhansky, N. Tokmagambetov.
Wave equation for operators with discrete spectrum and irregular propagation speed.
{\em to appear in Arch. Rational Mech. Anal.} 
https://arxiv.org/abs/1705.01418

\bibitem{RT}  M. Ruzhansky, V. Turunen.
 \emph{Pseudo-Differential Operators and Symmetries: Background Analysis and Advanced Topics},
 Basel, Birkh\"auser,
 2009.


\bibitem{T1986} M. E. Taylor.
\emph{Noncommutative harmonic analysis},
volume $22$ of Mathematical Surveys and Monographs, American Mathematical Society, 1986.

\bibitem{tER:97}
A.~F.~M. ter Elst, D.~W. Robinson.
\newblock Spectral estimates for positive {R}ockland operators.
\newblock In {\em Algebraic groups and {L}ie groups}, volume~9 of {\em Austral.
  Math. Soc. Lect. Ser.}, pages 195--213. Cambridge Univ. Press, Cambridge,
  1997.

\bibitem{thangavelu}
S. Thangavelu.
\emph{Harmonic analysis on the Heisenberg group},
Boston, Birkh\"auser,
1998.


\end{thebibliography}
\end{document}